\documentclass[b5paper,10pt,twoside,onecolumn,final]{article}

\usepackage{my_seamb}

\newcommand{\N}{\mathbb{N}}
\newcommand{\R}{\mathbb{R}}

\newcommand{\y}{\bar{y}}

\newcommand{\Ir}{{_{t}}\textsl{I}_b^\alpha}
\newcommand{\Ilc}{{_{a}}\textsl{I}_t^{1-\alpha}}
\newcommand{\Irc}{{_{t}}\textsl{I}_b^{1-\alpha}}
\newcommand{\Ilct}{{_{a}}\textsl{I}_{\tau}^{1-\alpha}}

\newcommand{\Dr}{{_{t}}\textsl{D}_b^\alpha}
\newcommand{\Dcl}{{^{C}_{a}}\textsl{D}_t^\alpha}
\newcommand{\Dclt}{{^{C}_{a}}\textsl{D}_{\tau}^\alpha}

\newcommand{\K}{K_P}
\newcommand{\Hl}{{_{a}}\textsl{J}_t^\alpha}
\newcommand{\Hr}{{_{t}}\textsl{J}_b^\alpha}

\newcommand{\Ilcp}{{_{a_i}}\textsl{I}_{t_i}^{1-\alpha_i}}
\newcommand{\Ircp}{{_{t}}\textsl{I}_{b_i}^{1-\alpha_i}}
\newcommand{\Drp}{{_{t_i}}\textsl{D}_{b_i}^{\alpha_i}}
\newcommand{\Dclp}{{^{C}_{a_i}}\textsl{D}_{t_i}^{\alpha_i}}

\newcommand{\PKi}{K_{P_{i}}}
\newcommand{\PAi}{A_{P_{i}}}
\newcommand{\PBi}{B_{P_{i}}}

\newcommand{\fonction}[5]{
\begin{array}[t]{lrcl}
#1 :&#2 & \longrightarrow & #3\\
& #4 & \longmapsto & #5
\end{array}
}

\begin{document}

\seambfrfp{2014}{38}{1}{\pageref{pgCAbdio}}

\vspace*{28mm}
\pagestyle{myheadings}
\markboth
{\hfill{\small\rm T. Odzijewicz and D. F. M. Torres}}
{{\small\rm The Generalized Fractional Calculus of Variations} \hfill}


\title{The Generalized Fractional Calculus of Variations}


\dzh

\author{Tatiana Odzijewicz \ and \ Delfim F. M. Torres}

\no{\smaller Center for Research and Development in Mathematics and Applications (CIDMA),
Department of Mathematics, University of Aveiro, 3810--193 Aveiro, Portugal.\\
Email: tatianao@ua.pt; delfim@ua.pt}


\dzh

\no{Received 4 January 2014\\
Accepted 7 January 2014}
\flh
\no{Communicated by K. P. Shum}
\flh
\no{\smallerbf AMS Mathematics Subject Classification (2010):} {\smaller\ 26A33, 49K05, 49K21}
\flh


\no{\smallerbf Abstract.}{\smaller\
We review the recent generalized fractional calculus of variations.
We consider variational problems containing generalized fractional integrals
and derivatives and study them using indirect methods. In particular,
we provide necessary optimality conditions of Euler--Lagrange type
for the fundamental and isoperimetric problems, natural boundary conditions,
and Noether type theorems.}
\zyh


\no{\smallerbf Keywords:}{\smaller\
Calculus of variations; Fractional calculus;
Euler--Lagrange equations;
Natural boundary conditions;
Isoperimetric problems;
Noether's theorem.}
\zjq
\normalsize


\section{Introduction}

Fractional differentiation means ``differentiation of arbitrary order''.
Its origin goes back more than 300 years, when in 1695 L'Hopital asked Leibniz the
meaning of $\frac{d^{n}y}{dx^{n}}$ for $n=\frac{1}{2}$.
After that, many famous mathematicians, like J. Fourier,
N. H. Abel, J. Liouville, B. Riemann, among others, contributed to
the development of Fractional Calculus \cite{hilfer,Podlubny,book:Samko}.
The theory of derivatives and integrals of arbitrary order
took more or less finished form by the end of the XIX century,
being very rich: fractional differentiation may be introduced
in several different ways, e.g., fractional derivatives of
Riemann--Liouville, Gr\"{u}nwald--Letnikov, Caputo, Miller--Ross, \ldots\
During three centuries, the theory of fractional derivatives
developed as a pure theoretical field of mathematics, useful
only for mathematicians. In the last few decades, however,
fractional differentiation proved very useful in various fields:
physics (classic and quantum mechanics, thermodynamics, etc.),
chemistry, biology, economics, engineering, signal and image processing,
and control theory \cite{book:Baleanu,TM,livro:ortigueira}.
Let
$$
{_aI_t^1}x(t) := \int_a^t x(\tau) d\tau.
$$
It is easy to prove, by induction, that
\begin{equation*}
{_aI_t^n}x(t) = \frac{1}{(n-1)!}\int_a^t (t-\tau)^{n-1}x(\tau) d\tau
\end{equation*}
for all $n \in \mathbb{N}$: if it is true for the $n$-fold integral,
then
\begin{equation*}
\begin{split}
{_aI_t^{n+1}}x(t)
&= {_aI_t^1} \left(\frac{1}{(n-1)!}\int_a^t (t-\tau)^{n-1}x(\tau) d\tau\right)\\
&= \int_a^t \left(\frac{1}{(n-1)!}\int_a^\xi (\xi-\tau)^{n-1}x(\tau) d\tau\right)d\xi.
\end{split}
\end{equation*}
Interchanging the order of integration gives
$$
{_aI_t^{n+1}}x(t) = \frac{1}{n!}\int_a^t (t-\tau)^{n}x(\tau) d\tau.
$$
The (left Riemann--Liouville) fractional integral of $x(t)$ of order
$\alpha > 0$, is then defined with the help of Euler's Gamma function $\Gamma$:
\begin{equation}
\label{eq:FI}
{_aI_t^{\alpha}}x(t) := \frac{1}{\Gamma(\alpha)}\int_a^t (t-\tau)^{\alpha-1}x(\tau) d\tau .
\end{equation}
Let $\alpha > 0$ and denote the fractional integral of $f$ of order $\alpha$ by
$$
{_aD_x^{-\alpha}} f(x) = \frac{1}{\Gamma(\alpha)} \int_a^x f(t) (x - t)^{\alpha-1} dt.
$$
If $m$ is the smallest integer exceeding $\alpha$,
then we define the \emph{fractional Riemann--Liouville derivative of $f$} of order $\alpha$ as
\begin{equation}
\label{eq:FD:RL}
\begin{split}
{_aD_x^{\alpha}} f(x) &= \frac{d^m}{d x^m} \left[
{_aD_x^{-(m-\alpha)}} f(x) \right]\\
&=\frac{1}{\Gamma(m - \alpha)} \frac{d^m}{d x^m} \int_a^x f(t) (x - t)^{m-\alpha-1} dt.
\end{split}
\end{equation}
Another definition of fractional derivatives
was introduced by M. Caputo in 1967, interchanging
the order of the operators $ \frac{d^m}{d x^m}$ and ${_aD_x^{-(m-\alpha)}}$
in \eqref{eq:FD:RL}:
\begin{equation}
\label{eq:FD:C}
{_a^CD_x^\alpha} := {_aD_x^{-(m-\alpha)}} \circ \frac{d^m}{dx^m}.
\end{equation}
Here we consider generalizations of operators \eqref{eq:FI}, \eqref{eq:FD:RL}
and \eqref{eq:FD:C} by considering more general kernels (see Section~\ref{sec:GFO}).

The classical fundamental problem of the calculus of variations
is formulated as follows: minimize (or maximize) the functional
\begin{equation*}
\mathcal{J}(x)=\int_a^b L(t,x(t), x'(t)) \, dt
\end{equation*}
on $\mathcal{D}=\{x\in C^1([a,b]) : x(a)=x_a,\, x(b)=x_b\}$,
where $L:[a,b]\times\mathbb{R}^{2n}\rightarrow \mathbb{R}$ is
twice continuously differentiable.
In Mechanics, function $L$ is called the \emph{Lagrangian};
functional $\mathcal{J}$ is called the \emph{Action}.
If $x$ gives a (local) minimum (or maximum)
to $\mathcal{J}$ on $\mathcal{D}$, then
\begin{equation*}
\frac{d}{dt} \partial_3 L\left(t,x(t), x'(t)\right)
= \partial_2 L\left(t,x(t),x'(t)\right)
\end{equation*}
holds for all $t\in[a,b]$, where we are using the notation
$\partial_i F$ for the partial derivative of a function $F$
with respect to its $i$th argument. This is the celebrated
Euler--Lagrange equation, which is a first-order
necessary optimality condition.
In Mechanics, if Lagrangian $L$ does not depend explicitly on $t$,
then the \emph{energy}
\begin{equation}
\label{eq:energ}
\mathcal{E}(x) := -L(x,x') + \frac{\partial L}{\partial x'}(x,x') \cdot x'
\end{equation}
is constant along physical trajectories $x$
(that is, along the solutions of the Euler--Lagrange equations).
Consider a particle of mass $m$, and let $x:\mathbb{R}\rightarrow \mathbb{R}^3$
denote the trajectory of this particle. Define the Lagrangian
to be the difference between the kinetic and potential energies,
$$
L(t,x,x') := T(x)-V(x) = \frac{1}{2} m\|x'\|^2-V(x),
$$
and the action of the trajectory from time $a$ to $b$ to be the integral
$$
\mathcal{J}(x)=\int_a^b L(t,x(t), x'(t)) \, dt .
$$
Hamilton's Principle of Least Action
asserts that particles follow trajectories which minimize the action.
Therefore, the solutions of the Euler--Lagrange equations
give the physical trajectories. In this case
the Euler--Lagrange equations give Newton's second law:
\begin{equation*}
m\frac{d^2x^i}{dt^2}=-\frac{\partial V}{\partial x^i}.
\end{equation*}
Let us consider the usual discretization
of a function
$f: t\in [a,b]\subset \R \mapsto f(t) \in \R$:
denote by $h=(b-a)/N$ the step of discretization;
consider the partition $t_k = a+ k h$, $k=0,\ldots ,N$, of $[a,b]$;
let $\mathbf{F} =\{ f_k :=f(t_k)\}_{k=0,\dots ,N}$;
and substitute the differential operator
$\frac{d}{dt} $ by $\Delta_+$ or $\Delta_-$:
\begin{equation*}
\begin{split}
\Delta_+ (\mathbf{F} ) &=\left\{ \frac{f_{k+1} -f_k}{h}
,\, 0\leq k\leq N-1\, ,\ 0 \right\}, \\
\Delta_- (\mathbf{F} )&=\left \{ 0, \frac{f_k - f_{k-1}}{h} ,
\ 1\leq k \leq N \right \} .
\end{split}
\end{equation*}
The discrete version of the Euler--Lagrange equation obtained by
the direct embedding is
\begin{equation}
\label{eq:NS}
\frac{x_{k+2}-2 x_{k+1}+x_k}{h^2}m + \frac{\partial V}{\partial x}(x_k) = 0,
\quad k = 0, \ldots, N-2,
\end{equation}
where $N = \frac{b-a}{h}$ and $x_k = x(a+k h)$.
This numerical scheme is of order one:
we make an error of order $h$ at each step,
which is of course not good. We can do better
by considering the variational structure of the problem.
All Lagrangian systems possess a variational structure, \textrm{i.e.},
their solutions correspond to critical points of a functional
and this characterization does not depend on the system coordinates.
This induces strong constraints on solutions.
For the example we are considering, which is autonomous,
the \emph{conservation of energy} asserts that
$$
\mathcal{E}(x)=T(x)+V(x) = const.
$$
Using such conservation law we can easily improve the numerical scheme
\eqref{eq:NS} into a new one with an error of order $h^2$ at each step,
which is of course better. Unfortunately, in real systems friction corrupts conservation of energy,
and the usefulness of variational principles is lost: ``forces of a frictional
nature are outside the realm of variational principles''.
For conservative systems, variational methods are equivalent to the original method used by Newton.
However, while Newton's equations allow nonconservative forces, the
later techniques of Lagrangian and Hamiltonian mechanics have no
direct way to dealing with them. Let us recall the classical problem of linear friction:
\begin{equation}
\label{eq:L:fric}
m\frac{d^2x}{dt^2}+\gamma\frac{dx}{dt} -\frac{\partial V}{\partial x}=0, \quad \gamma>0.
\end{equation}
In 1931, Bauer proved that it is impossible to use a
variational principle to derive a single linear dissipative
equation of motion with constant coefficients like \eqref{eq:L:fric}.
Bauer's theorem expresses the well-known belief that there is no direct
method of applying variational principles to nonconservative
systems, which are characterized by friction or other dissipative
processes. Fractional derivatives provide an elegant solution to the problem.
Indeed, the proof of Bauer's theorem relies on the tacit
assumption that all derivatives are of integer order.
If a Lagrangian is constructed using fractional (noninteger order) derivatives,
then the resulting equation of motion can be nonconservative!
This was first proved by F. Riewe in 1996/97 \cite{CD:Riewe:1996,CD:Riewe:1997},
marking the beginning of the \emph{Fractional Calculus of Variations} (FCV).
Because most processes observed in the physical world are
nonconservative, FCV constitutes an important research area,
allowing to apply the power of variational methods to real systems.
The first book on the subject is \cite{book:AD}, which provides
a gentle introduction to the FCV. The model problem considered in \cite{book:AD}
is to find an admissible function giving a minimum value
to an integral functional that depends on an unknown function (or functions), of one
or several variables, and its fractional derivatives and/or
fractional integrals. Here we explain how the main results presented
in \cite{book:AD} can be extended by considering generalized fractional operators
\cite{Odz:PhD}.

The text is organized as follows. Section~\ref{sec:Prelim}
recalls the definitions of generalized fractional operators,
for functions of one (Section~\ref{sec:GFO})
and several variables (Section~\ref{sub:sec:PO}).
Main results are then given in Section~\ref{sec:MR}.
Section~\ref{sec:fund} considers the one-dimensional
fundamental problem of the calculus of variations
with generalized fractional operators,
providing an appropriate Euler--Lagrange equation.
Next, in Section~\ref{sec:nat}, we study variational problems with free
end points and, besides Euler--Lagrange equations,
we prove the so called natural boundary conditions (transversality conditions).
As particular cases, we obtain natural boundary conditions for problems with standard
fractional operators \eqref{eq:FI}--\eqref{eq:FD:C}.
Section~\ref{sec:iso} is devoted to generalized fractional
isoperimetric problems. We aim to find functions that minimize an integral functional
subject to given boundary conditions and isoperimetric constraints. We prove necessary
optimality conditions and, as corollaries, we obtain Euler--Lagrange equations
for isoperimetric problems with standard fractional operators \eqref{eq:FI}--\eqref{eq:FD:C}.
Furthermore, we illustrate our results through an example. In Section~\ref{sec:NTH:sing}
we prove a generalized fractional counterpart of Noether's theorem. Assuming invariance
of the functional, we prove that any extremal must satisfy a certain generalized
fractional equation. Finally, in Section~\ref{sec:SEV} we study multidimensional
fractional variational problems with generalized partial operators. We end with
Section~\ref{sec:conc} of conclusions.


\section{Preliminaries}
\label{sec:Prelim}

This section presents definitions of generalized fractional operators.
In special cases, these operators simplify to the classical Riemann--Liouville
fractional integrals \eqref{eq:FI},
and Riemann--Liouville \eqref{eq:FD:RL}
and Caputo fractional derivatives \eqref{eq:FD:C}.


\subsection{Generalized fractional operators}
\label{sec:GFO}

Let us define the following triangle:
\begin{equation*}
\Delta:=\left\{(t,\tau)\in\R^2:~a\leq \tau<t\leq b\right\}.
\end{equation*}

\begin{definition}
Let us consider a function $k:\Delta\rightarrow\R$.
For any function $f:(a,b)\rightarrow\R$, the generalized
fractional integral operator $K_P$ is defined
for almost all $t \in (a,b)$ by
\begin{equation*}
K_{P}[f](t) = \lambda \int_a^t k(t,\tau) f(\tau)d\tau
+ \mu \int_t^b k(\tau,t) f(\tau) d\tau
\end{equation*}
with $P=\langle a,t,b,\lambda,\mu \rangle$, $\lambda$, $\mu\in \R$.
\end{definition}

In particular, for suitably chosen kernels $k(t,\tau)$ and sets $P$,
kernel operators $\K$ reduce to the classical or variable order
fractional integrals of Riemann--Liouville type,
and classical fractional integrals of Hadamard type.

\example
\begin{enumerate}
\item Let $k^{\alpha}(t-\tau)
=\frac{1}{\Gamma(\alpha)}(t-\tau)^{\alpha-1}$ and $0<\alpha<1$.
If $P=\langle a,t,b,1,0\rangle$, then
\begin{equation*}
K_{P}[f](t)=\frac{1}{\Gamma(\alpha)}
\int\limits_a^t(t-\tau)^{\alpha-1}f(\tau)d\tau
=: {_{a}}\textsl{I}^{\alpha}_{t} [f](t)
\end{equation*}
gives the left Riemann--Liouville fractional integral
of order $\alpha$. Now let $P=\langle a,t,b,0,1\rangle$. Then,
\begin{equation*}
K_{P}[f](t)=\frac{1}{\Gamma(\alpha)}
\int\limits_t^b(\tau-t)^{\alpha-1}f(\tau)d\tau
=: {_{t}}\textsl{I}^{\alpha}_{b} [f](t)
\end{equation*}
is the right Riemann--Liouville fractional integral
of order $\alpha$.

\item For $k^{\alpha}(t,\tau)
=\frac{1}{\Gamma(\alpha(t,\tau))}(t-\tau)^{\alpha(t,\tau)-1}$
and $P=\langle a,t,b,1,0\rangle$,
\begin{equation*}
K_{P}[f](t)=
\int\limits_a^t\frac{1}{\Gamma(\alpha(t,\tau)}(t-\tau)^{\alpha(t,\tau)-1}f(\tau)d\tau
=: {_{a}}\textsl{I}^{\alpha(\cdot,\cdot)}_{t} [f](t)
\end{equation*}
is the left Riemann--Liouville fractional integral
of variable order $\alpha(\cdot,\cdot)$,
and for $P=\langle a,t,b,0,1\rangle$
\begin{equation*}
K_{P}[f](t)=
\int\limits_t^b\frac{1}{\Gamma(\alpha(\tau,t))}(\tau-t)^{\alpha(t,\tau)-1}f(\tau)d\tau
=: {_{t}}\textsl{I}^{\alpha(\cdot,\cdot)}_{b} [f](t)
\end{equation*}
is the right Riemann--Liouville fractional integral
of variable order $\alpha(\cdot,\cdot)$.

\item If $0<\alpha<1$, kernel $k^{\alpha}(t,\tau)
=\frac{1}{\Gamma(\alpha)}\left(\log\frac{t}{\tau}\right)^{\alpha-1}\frac{1}{\tau}$
and $P=\langle a,t,b,1,0\rangle$, then the general operator $K_{P}$
reduces to the left Hadamard fractional integral:
\begin{equation*}
K_{P}[f](t)=
\frac{1}{\Gamma(\alpha)}\int_a^t
\left(\log\frac{t}{\tau}\right)^{\alpha-1}\frac{f(\tau)d\tau}{\tau}
=: \Hl [f](t);
\end{equation*}
while for $P=\langle a,t,b,0,1\rangle$ operator $K_{P}$
reduces to the right Hadamard fractional integral:
\begin{equation*}
K_{P}[f](t)=
\frac{1}{\Gamma(\alpha)}\int_t^b
\left(\log\frac{\tau}{t}\right)^{\alpha-1}\frac{f(\tau)d\tau}{\tau}
=:\Hr [f](t).
\end{equation*}

\item Generalized fractional integrals can be also reduced
to other fractional operators, e.g.,
Riesz, Katugampola or Kilbas operators. Their definitions
can be found in \cite{Katugampola,Kilbas,book:Kilbas}.
\end{enumerate}
\dl

The generalized differential operators $A_P$
and $B_P$ are defined with the help of operator $K_P$.

\begin{definition}
The generalized fractional derivative
of Riemann--Liouville type, denoted by $A_P$,
is defined by
$$
A_P = \frac{d}{dt}\circ K_P.
$$
\end{definition}

The next differential operator is obtained by interchanging the order
of the operators in the composition that defines $A_P$.

\begin{definition}
The general kernel differential operator of Caputo type, denoted by $B_P$,
is given by
$$
B_P =K_P \circ\frac{d}{dt}.
$$
\end{definition}

\example
The standard Riemann--Liouville \eqref{eq:FD:RL}
and Caputo \eqref{eq:FD:C} fractional derivatives
(see, \textrm{e.g.}, \cite{book:Kilbas,book:Klimek}) are easily obtained
from the general kernel operators $A_P $ and $B_P$, respectively.
Let $k^{\alpha}(t-\tau)=\frac{1}{\Gamma(1-\alpha)}(t-\tau)^{-\alpha}$,
$\alpha \in (0,1)$. If $P=\langle a,t,b,1,0\rangle$, then
\begin{equation*}
A_{P} [f](t)=\frac{1}{\Gamma(1-\alpha)}
\frac{d}{dt} \int\limits_a^t(t-\tau)^{-\alpha}f(\tau)d\tau
=: {_{a}}\textsl{D}^{\alpha}_{t} [f](t)
\end{equation*}
is the standard left Riemann--Liouville fractional derivative
of order $\alpha$, while
\begin{equation*}
B_{P} [f](t)=\frac{1}{\Gamma(1-\alpha)}
\int\limits_a^t(t-\tau)^{-\alpha} f'(\tau)d\tau
=: {^{C}_{a}}\textsl{D}^{\alpha}_{t} [f](t)
\end{equation*}
is the standard left Caputo fractional derivative of order $\alpha$;
if $P=\langle a,t,b,0,1\rangle$, then
\begin{equation*}
- A_{P} [f](t)
=- \frac{1}{\Gamma(1-\alpha)} \frac{d}{dt}
\int\limits_t^b(\tau-t)^{-\alpha}f(\tau)d\tau
=: {_{t}}\textsl{D}^{\alpha}_{b} [f](t)
\end{equation*}
is the standard right Riemann--Liouville
fractional derivative of order $\alpha$, while
\begin{equation*}
- B_{P} [f](t) = - \frac{1}{\Gamma(1-\alpha)}
\int\limits_t^b(\tau-t)^{-\alpha} f'(\tau)d\tau
=: {^{C}_{t}}\textsl{D}^{\alpha}_{b} [f](t)
\end{equation*}
is the standard right Caputo fractional derivative of order $\alpha$.
\dl


\subsection{Generalized partial fractional operators}
\label{sub:sec:PO}

In this section, we introduce notions of generalized partial fractional integrals
and derivatives, in a multidimensional finite domain. They are natural generalizations
of the corresponding fractional operators in the single variable case. Furthermore,
similarly as in the integer order case, computation of partial fractional derivatives
and integrals is reduced to the computation of one-variable fractional operators.
Along the work, for $i=1,\dots,n$, let $a_i,b_i$ and $\alpha_i$ be numbers in
$\R$ and $t=(t_1,\dots,t_n)$ be such that $t\in \Omega_n$, where
$\Omega_n=(a_1,b_1)\times\dots\times(a_n,b_n)$ is a subset of $\R^n$.
Moreover, let us define the following sets:
\begin{equation*}
\Delta_i:=\left\{(t_i,\tau)\in\R^2:~a_i
\leq \tau <t_i\leq b_i\right\},~i=1\dots,n.
\end{equation*}
Let us assume that $\lambda=(\lambda_1,\dots,\lambda_n)$ and
$\mu=(\mu_1,\dots,\mu_n)$ are in $\mathbb{R}^n$. We shall present
definitions of generalized partial fractional integrals and derivatives.
Let $k_i:\Delta_i\rightarrow\R$, $i=1\dots,n$, and $t\in\Omega_n$.

\begin{definition}
For any function $f:\Omega_n\rightarrow\R$, the generalized
partial integral $K_{P_i}$ is defined, for almost
all $t_i \in (a_i,b_i)$, by
\begin{multline*}
K_{P_{i}}[f](t):=\lambda_i\int\limits_{a_i}^{t_i}
k_i(t_i,\tau)f(t_1,\dots,t_{i-1},\tau,t_{i+1},\dots,t_n)d\tau \\
+\mu_i\int\limits_{t_i}^{b_i}k_i(\tau,t_i)f(t_1,\dots,t_{i-1},\tau,t_{i+1},
\dots,t_n)d\tau,
\end{multline*}
where $P_{i}=\langle a_i,t_i,b_i,\lambda_i,\mu_i \rangle $.
\end{definition}

\begin{definition}
The generalized partial fractional derivative of Riemann--Liouville
type with respect to the $i$th variable $t_i$ is given by
\begin{equation*}
A_{P_{i}}:=\frac{\partial}{\partial t_i}\circ K_{P_{i}}.
\end{equation*}
\end{definition}

\begin{definition}
The generalized partial fractional derivative of Caputo type
with respect to the $i$th variable $t_i$ is given by
\begin{equation*}
B_{P_{i}}:=K_{P_{i}}\circ\frac{\partial}{\partial t_i}.
\end{equation*}
\end{definition}

\example
Similarly as in the one-dimensional case, partial operators $K$, $A$ and $B$
reduce to the standard partial fractional integrals and derivatives.
The left- or right-sided Riemann--Liouville partial fractional integral
with respect to the $i$th variable $t_i$ is obtained by choosing the kernel
$k_i^{\alpha}(t_i,\tau)=\frac{1}{\Gamma(\alpha_i)}(t_i-\tau)^{\alpha_i-1}$:
\begin{equation*}
K_{P_{i}}[f](t)=\frac{1}{\Gamma(\alpha_i)}\int\limits_{a_i}^{t_i}(t_i-\tau)^{\alpha_i-1}
f(t_1,\dots,t_{i-1},\tau,t_{i+1},\dots,t_n)d\tau \\
=: {_{a_i}}\textsl{I}^{\alpha_i}_{t_i} [f](t)
\end{equation*}
for $P_{i}=\langle a_i,t_i,b_i,1,0\rangle$, and
\begin{equation*}
K_{P_{i}}[f](t)=\frac{1}{\Gamma(\alpha_i)}\int\limits_{t_i}^{b_i}
(\tau-t_i)^{\alpha_i-1}f(t_1,\dots,t_{i-1},\tau,t_{i+1},\dots,t_n)d\tau\\
=: {_{t_i}}\textsl{I}^{\alpha_i}_{b_i} [f](t)
\end{equation*}
for $P_{i}=\langle a_i,t_i,b_i,0,1\rangle$. The standard left- and
right-sided Riemann--Liouville and Caputo partial fractional derivatives
with respect to $i$th variable $t_i$ are obtained by choosing
$k_i^{\alpha}(t_i,\tau)=\frac{1}{\Gamma(1-\alpha_i)}(t_i-\tau)^{-\alpha_i}$.
If $P_{i}=\langle a_i,t_i,b_i,1,0\rangle$, then
\begin{equation*}
\begin{split}
A_{P_{i}}[f](t)&=\frac{1}{\Gamma(1-\alpha_i)}\frac{\partial}{\partial t_i}
\int\limits_{a_i}^{t_i}(t_i-\tau)^{-\alpha_i}f(t_1,\dots,t_{i-1},\tau,t_{i+1},
\dots,t_n)d\tau\\
&=:{_{a_i}}\textsl{D}^{\alpha_i}_{t_i} [f](t),
\end{split}
\end{equation*}
\begin{equation*}
\begin{split}
B_{P_{i}}[f](t)&=\frac{1}{\Gamma(1-\alpha_i)}\int\limits_{a_i}^{t_i}
(t_i-\tau)^{-\alpha_i}\frac{\partial}{\partial \tau}
f(t_1,\dots,t_{i-1},\tau,t_{i+1},\dots,t_n)d\tau\\
&=:{^{C}_{a_i}}\textsl{D}^{\alpha_i}_{t_i} [f](t).
\end{split}
\end{equation*}
If $P_{i}=\langle a_i,t_i,b_i,0,1\rangle$, then
\begin{equation*}
\begin{split}
-A_{P_{i}}[f](t)&=\frac{-1}{\Gamma(1-\alpha_i)}\frac{\partial}{\partial t_i}
\int\limits_{t_i}^{b_i}(\tau-t_i)^{-\alpha_i}f(t_1,
\dots,t_{i-1},\tau,t_{i+1},\dots,t_n)d\tau\\
&=:{_{t_i}}\textsl{D}^{\alpha_i}_{b_i} [f](t),
\end{split}
\end{equation*}
\begin{equation*}
\begin{split}
-B_{P_{i}}[f](t)&=\frac{-1}{\Gamma(1-\alpha_i)}\int\limits_{t_i}^{b_i}
(\tau-t_i)^{-\alpha_i}\frac{\partial}{\partial \tau}
f(t_1,\dots,t_{i-1},\tau,t_{i+1},\dots,t_n)d\tau \\
&=:{^{C}_{t_i}}\textsl{D}^{\alpha_i}_{b_i} [f](t).
\end{split}
\end{equation*}
Moreover, one can easily check that variable order partial fractional
integrals and derivatives are also particular cases of operators $\PKi$, $\PAi$
and $\PBi$. Definitions of variable order partial fractional
operators can be found in \cite{Ja}.
\dl


\section{The general fractional calculus of variations}
\label{sec:MR}

Fractional Variational Calculus (FVC) was first introduced in 1996,
by Fred Riewe \cite{CD:Riewe:1996},
and the reason is well explained by Lanczos, who wrote:
``Forces of a frictional nature are outside the realm of variational principles''.
The idea of FVC is to unify the calculus of variations and the fractional calculus
by inserting fractional derivatives (and/or integrals) into the variational functionals.
With less than twenty years of existence, FVC was developed through several different
approaches --- results include problems depending on Caputo fractional derivatives,
Riemann--Liouville fractional derivatives, Riesz fractional derivatives,
Hadamard fractional derivatives, and variable order fractional derivatives \cite{book:AD}.
The Generalized Fractional Calculus of Variations (GFCV)
concerns operators depending on general kernels, unifying different perspectives to the FVC.
As particular cases, such operators reduce to the standard
fractional integrals and derivatives (see, e.g.,
\cite{OmPrakashAgrawal,Lupa,MyID:226,FVC_Gen_Int,FVC_Sev,GreenThm,NoetherGen}).
Before presenting the GFCV, we define the concept of minimizer.
Let $(X,\left\|\cdot\right\|)$ be a normed linear space and $\mathcal{I}$ be a functional defined
on a nonempty subset $\mathcal{A}$ of $X$. Moreover, let us introduce the following set:
if $\y\in\mathcal{A}$ and $\delta>0$, then
\begin{equation*}
\mathcal{N}_{\delta}(\y):=\left\{y\in\mathcal{A}:\left\|y-\y\right\|<\delta\right\}
\end{equation*}
is called a neighborhood of $\y$ in $\mathcal{A}$.

\begin{definition}
Function $\y\in\mathcal{A}$ is called a minimizer of $\mathcal{I}$ if there exists
a neighborhood $\mathcal{N}_{\delta}(\y)$ of $\y$ such that
$\mathcal{I}(\y)\leq\mathcal{I}(y)$ for all $y\in \mathcal{N}_{\delta}(\y)$.
\end{definition}

Note that any function $y\in\mathcal{N}_{\delta}(\y)$ can be represented,
in a convenient way, as a perturbation of $\y$. Precisely,
\begin{equation*}
\forall y\in\mathcal{N}_{\delta}(\y),~~\exists\eta
\in\mathcal{A}_0,~~ y=\y+h\eta,~~\left|h\right|\leq\varepsilon,
\end{equation*}
where $0<\varepsilon<\frac{\delta}{\left\|\eta\right\|}$ and $\mathcal{A}_0$
is a suitable set of functions $\eta$ such that
\begin{equation*}
\mathcal{A}_0=\left\{\eta\in X:\y+h\eta\in\mathcal{A},
~~\left|h\right|\leq\varepsilon\right\}.
\end{equation*}


\subsection{Fundamental problem}
\label{sec:fund}

For $P=\langle a,t,b,\lambda,\mu\rangle$,
let us consider the following functional:
\begin{equation}
\label{eq:F:1}
\fonction{\mathcal{I}}{\mathcal{A}(y_a,y_b)}{\R}{y}{
\displaystyle\int\limits_a^b
F(y(t),K_P[y](t),\dot{y}(t),B_P[y](t),t) \; dt ,}
\end{equation}
where
\begin{multline*}
\mathcal{A}(y_a,y_b):=\left\{y\in C^1([a,b];\R):\;
y(a)=y_a,~y(b)=y_b,\; \right.\\
\left. \textnormal{and}\;
K_P[y],B_P[y]\in C([a,b];\R)\right\},
\end{multline*}
$\dot{y}$ denotes the classical derivative of $y$,
$K_P$ is the generalized fractional integral operator
with kernel belonging to $L^q(\Delta;\R)$,
$B_P=K_P\circ\frac{d}{dt}$, and $F$ is the Lagrangian function
\begin{equation*}
\fonction{F}{\R^4 \times [a,b]}{\R}{(x_1,x_2,x_3,x_4,t)}{F(x_1,x_2,x_3,x_4,t)}
\end{equation*}
of class $C^1$. Moreover, we assume that
\begin{itemize}
\item $K_{P^*}\left[\tau\mapsto\partial_2
F(y(\tau),K_P[y](\tau),\dot{y}(\tau),B_P[y](\tau),\tau)\right]\in C([a,b];\R)$,
\item $t\mapsto\partial_3 F (y(t),K_P[y](t),\dot{y}(t),B_P[y](t),t)\in C^1([a,b];\R)$,
\item $K_{P^*}\left[\tau\mapsto\partial_4
F(y(\tau),K_P[y](\tau),\dot{y}(\tau),B_P[y](\tau),\tau)\right]\in C^1([a,b];\R)$,
\end{itemize}
where $P^*$ is the dual set of $P$, that is, $P^*=\langle a,t,b,\mu,\lambda\rangle$.

The next result gives a necessary optimality condition of Euler--Lagrange type
for the problem of finding a function minimizing functional \eqref{eq:F:1}.

\begin{thm}
{\rm (cf. \cite{MyID:226})}
Let $\bar{y}\in\mathcal{A}(y_a,y_b)$ be a minimizer of functional \eqref{eq:F:1}.
Then, $\bar{y}$ satisfies the following Euler--Lagrange equation:
\begin{multline}
\label{eq:EL:OCM}
\frac{d}{dt}\left[\partial_3 F\left(\star_{\bar{y}}\right)(t)\right]
+A_{P^*}\left[\tau\mapsto\partial_4 F\left(\star_{\bar{y}}\right)(\tau)\right](t)\\
=\partial_1 F\left(\star_{\bar{y}}\right)(t)
+K_{P^*}\left[\tau\mapsto\partial_2 F\left(\star_{\bar{y}}\right)(\tau)\right](t),
\end{multline}
where $\left(\star_{\bar{y}}\right)(t)
=\left(\y(t),K_P[\y](t),\dot{\y}(t),B_P[\y](t),t\right)$, $t\in (a,b)$.
\end{thm}

From now on, in order to simplify the notation,
for $T$ and $S$ being fractional operators, we write shortly
$$
T\left[\partial_i F(y(\tau),T[y](\tau),\dot{y}(\tau),S[y](\tau),\tau)\right]
$$
instead of
$$
T\left[\tau\mapsto\partial_i
F(y(\tau),T[y](\tau),\dot{y}(\tau),S[y](\tau),\tau)\right],
$$
$i=1,\dots,5$.

\example
{\rm (cf. \cite{MyID:226})}
Let $P=\langle 0,t,1,1,0 \rangle$.
Consider minimizing
\begin{equation*}
\mathcal{I}(y)=\int\limits_0^1 \left(K_P[y](t)+t\right)^2\; dt
\end{equation*}
subject to given boundary conditions
\begin{equation*}
y(0)=-1~~\textnormal{and}~~y(1)=-1-\int\limits_0^1u(1-\tau)\;d\tau,
\end{equation*}
where the kernel $k$ of $K_P$ is such that $k(t,\tau)=h(t-\tau)$ with
$h\in C^1([0,1];\R)$ and $h(0)=1$. Here the resolvent $u$ is related
to the kernel $h$ by
$$
u(t)=\mathcal{L}^{-1}\left[\frac{1}{s\tilde{h}(s)}-1\right](t),
\quad \tilde{h}(s)=\mathcal{L}[h](s),
$$
where $\mathcal{L}$ and $\mathcal{L}^{-1}$
are the direct and the inverse Laplace operators, respectively. We apply
Theorem~3.2 with Lagrangian $F$ given by
$$
F(x_1,x_2,x_3,x_4,t)=(x_2+t)^2.
$$
Because
\begin{equation*}
y(t)=-1-\int\limits_0^t u(t-\tau)\;d\tau
\end{equation*}
is the solution to the Volterra integral equation
\begin{equation*}
K_P[y](t)+t=0
\end{equation*}
of the first kind (see, e.g., Eq.~16, p.~114 of \cite{book:Polyanin}),
it satisfies our generalized Euler--Lagrange equation \eqref{eq:EL:OCM}, that is,
\begin{equation*}
K_{P^*}\left[K_P[y](\tau)+\tau\right](t)=0,
\quad t\in (a,b).
\end{equation*}
In particular, for kernel $h(t-\tau)=e^{-(t-\tau)}$ and boundary conditions
$y(0)=-1$ and $y(1)=-2$, the solution is $y(t)=-1-t$.
\dl

\begin{cor}
Let $0<\alpha<\frac{1}{q}$ and let $\bar{y}\in C^1([a,b];\R)$
be a solution to the problem of minimizing the functional
\begin{equation*}
\mathcal{I}(y)=\int\limits_a^b\;
F\left(y(t),\Ilc [y](t),\dot{y}(t),\Dcl [y](t),t\right)\; dt
\end{equation*}
subject to the boundary conditions $y(a)=y_a$ and $y(b)=y_b$, where
\begin{itemize}
\item $F\in C^1(\R^4\times [a,b];\R)$,

\item functions $t\mapsto\partial_1 F\left(y(t),\Ilc [y](t),\dot{y}(t),\Dcl [y](t),t\right)$ and
$$
\Irc\left[\partial_2 F(y(\tau),\Ilct [y](\tau),\dot{y}(\tau),\Dclt [y](\tau),\tau)\right]
$$
are continuous on $[a,b]$,

\item functions $t\mapsto\partial_3 F\left(y(t),\Ilc [y](t),\dot{y}(t),\Dcl [y](t),t\right)$ and
$$
\Irc\left[\partial_4 F(y(\tau),\Ilct [y](\tau),\dot{y}(\tau),\Dclt [y](\tau),\tau)\right]
$$
are continuously differentiable on $[a,b]$.
\end{itemize}
Then, the Euler--Lagrange equation
\begin{multline*}
\frac{d}{dt}\left(\partial_3 F(\bar{y}(t),\Ilc [\bar{y}](t),
\dot{\bar{y}}(t),\Dcl [\bar{y}](t),t)\right)\\
-\Dr\left[\partial_4 F(\bar{y}(\tau),\Ilct [\bar{y}](\tau),
\dot{\bar{y}}(\tau),\Dclt [\bar{y}](\tau),\tau)\right](t)\\
=\partial_1 F(\bar{y}(t),\Ilc [\bar{y}](t),
\dot{\bar{y}}(t),\Dcl [\bar{y}](t),t)\\
+\Ir\left[\partial_2 F(\bar{y}(\tau),\Ilct [\bar{y}](\tau),
\dot{\bar{y}}(\tau),\Dclt [\bar{y}](\tau),\tau)\right](t)
\end{multline*}
holds for $t\in (a,b)$.
\end{cor}


\subsection{Free initial boundary}
\label{sec:nat}

Let us define the set
\begin{multline*}
\mathcal{A}(y_b):=\left\{y\in C^1 ([a,b];\R):\; y(a)\;\textnormal{is free},
~y(b)=y_b,\; \right.\\
\left. \textnormal{and}\; K_P[y],B_P[y]\in C([a,b];\R)\right\},
\end{multline*}
and let $\y$ be a minimizer of functional
\eqref{eq:F:1} on $\mathcal{A}(y_b)$, i.e., $\y$ minimizes
\begin{equation}
\label{eq:F:2}
\fonction{\mathcal{I}}{\mathcal{A}(y_b)}{\R}{y}{\displaystyle\int\limits_a^b
F(y(t),K_P[y](t),\dot{y}(t),B_P[y](t),t) \; dt .}
\end{equation}

The next result shows that if in the generalized fractional variational problem
one initial boundary is not preassigned,
then a certain natural boundary condition can be determined.
\begin{thm}
{\rm (cf. \cite{MyID:226})}
If $\y\in\mathcal{A}(y_b)$ is a solution to the problem of minimizing functional
\eqref{eq:F:2} on the set $\mathcal{A}(y_b)$, then $\y$ satisfies
the Euler--Lagrange equation \eqref{eq:EL:OCM}. Moreover, the extra boundary condition
\begin{equation*}
\left.\partial_3 F(\star_{\bar{y}})(t)\right|_a
+\left.K_{P^*}\left[\partial_4 F(\star_{\y})(\tau)\right](t)\right|_a=0
\end{equation*}
holds with $(\star_{\bar{y}})(t)=(\bar{y}(t),K_P[\bar{y}](t),\dot{\bar{y}}(t),B_P[\bar{y}](t),t)$.
\end{thm}

\begin{cor}
{\rm (cf. \cite{OmPrakashAgrawal3})}
Let $0<\alpha<\frac{1}{q}$ and $\mathcal{I}$ be the functional given by
\begin{equation*}
\mathcal{I}(y)=\int\limits_a^b
F\left(y(t),_{a}^{C}\textsl{D}_t^\alpha [y](t),t\right)dt,
\end{equation*}
where $F\in C^1(\R^2\times [a,b];\R)$ and
$\Ilc\left[\partial_2 F\left(y(\tau),_{a}^{C}\textsl{D}_{\tau}^\alpha [y](\tau),\tau\right)\right]$
is continuously differentiable on $[a,b]$. If $\y\in C^1([a,b];\R)$ is a minimizer of $\mathcal{I}$
among all functions satisfying the boundary condition $y(b)=y_b$,
then $\y$ satisfies the Euler--Lagrange equation
\begin{equation*}
\partial_1 F\left(\y(t),_{a}^{C}\textsl{D}_t^\alpha [\y](t),t\right)
+_{t}\textsl{D}_b^\alpha\left[\partial_2 F\left(\y(\tau),
_{a}^{C}\textsl{D}_{\tau}^\alpha [\y](\tau),\tau\right)\right](t)=0
\end{equation*}
for all $t\in (a,b)$ and the fractional natural boundary condition
\begin{equation*}
\left._{t}\textsl{I}_b^{1-\alpha}\left[\partial_2
F\left(\y(\tau),_{a}^{C}\textsl{D}_{\tau}^\alpha [\y](\tau),
\tau\right)\right](t)\right|_{a}=0.
\end{equation*}
\end{cor}


\subsection{Isoperimetric problem}
\label{sec:iso}

Let $P=\langle a,t,b,\lambda,\mu\rangle$.
We now define the following functional:
\begin{equation}
\label{eq:F:3}
\fonction{\mathcal{J}}{\mathcal{A}(y_a,y_b)}{\R}{y}{\displaystyle \int\limits_a^b
G\left(y(t),K_P[y](t),\dot{y}(t),B_P[y](t),t\right) \; dt ,}
\end{equation}
where by $\dot{y}$ we understand the classical derivative of $y$, $K_P$
is the generalized fractional integral operator with kernel belonging to
$L^q(\Delta;\R)$, $B_P=K_P\circ\frac{d}{dt}$, and $G$ is a Lagrangian
\begin{equation*}
\fonction{G}{\R^4 \times [a,b]}{\R}{(x_1,x_2,x_3,x_4,t)}{G(x_1,x_2,x_3,x_4,t)}
\end{equation*}
of class $C^1$. Moreover, we assume that
\begin{itemize}
\item $K_{P^*}\left[\partial_2 G (y(\tau),K_P[y](\tau),\dot{y}(\tau),B_P[y](\tau),\tau)\right]
\in C([a,b];\R)$,

\item $t\mapsto\partial_3 G (y(t),K_P[y](t),\dot{y}(t),B_P[y](t),t)
\in C^1([a,b];\R)$,

\item $K_{P^*}\left[\partial_4 G (y(\tau),K_P[y](\tau),\dot{y}(\tau),B_P[y](\tau),\tau)\right]
\in C^1([a,b];\R)$.
\end{itemize}
The main problem considered in this section consists to find a minimizer of functional \eqref{eq:F:1}
subject to the isoperimetric constraint $\mathcal{J}(y)=\xi$. In order to deal with this type
of problems, in the next theorem we provide a necessary optimality condition.

\begin{thm}
{\rm (cf. \cite{MyID:226})}
Suppose that $\y$ is a minimizer of functional $\mathcal{I}$ in the class
$$
\mathcal{A}_{\xi}(y_a,y_b):=\left\{y\in\mathcal{A}(y_a,y_b):\mathcal{J}(y)=\xi\right\}.
$$
Then there exists a real constant $\lambda_0$ such that, for $H=F-\lambda_0 G$, equation
\begin{multline}
\label{eq:EL:ISO1}
\frac{d}{dt}\left[\partial_3 H(\star_{\bar{y}})(t)\right]
+A_{P^*}\left[\partial_4 H(\star_{\bar{y}})(\tau)\right](t)\\
=\partial_1 H(\star_{\bar{y}})(t)+K_{P^*}\left[\partial_2
H(\star_{\bar{y}})(\tau)\right](t)
\end{multline}
holds for $t\in (a,b)$, provided
\begin{multline*}
\frac{d}{dt}\left[\partial_3 G(\star_{\bar{y}})(t)\right]
+A_{P^*}\left[\partial_4 G(\star_{\bar{y}})(\tau)\right](t)\\
\neq\partial_1 G(\star_{\bar{y}})(t)+K_{P^*}\left[\partial_2
G(\star_{\bar{y}})(\tau)\right](t),~t\in (a,b),
\end{multline*}
where $(\star_{\bar{y}})(t)=(\bar{y}(t),
K_P[\bar{y}](t),\dot{\bar{y}}(t),B_P[\bar{y}](t),t)$.
\end{thm}

\example
{\rm (cf. \cite{MyID:226})}
Let $P=\langle 0,t,1,1,0\rangle$.
Consider the problem
\begin{equation*}
\begin{gathered}
\mathcal{I}(y)=\int\limits_0^1\left({\textsl{K}_P} [y](t)
+t\right)^2dt \longrightarrow \min,\\
\mathcal{J}(y)=\int\limits_0^1 t{\textsl{K}_P} [y](t)\;dt = \xi,\\
y(0)=\xi-1, \quad y(1)=(\xi-1)\left(1
+\int\limits_0^1 u(1-\tau) d\tau\right),
\end{gathered}
\end{equation*}
where the kernel $k$ is such that $k(t,\tau)=h(t-\tau)$ with $h\in C^1([0,1];\R)$,
$h(0)=1$ and $\textsl{K}_{P^*}[id](t)\neq 0$ ($id$ stands for the identity transformation,
i.e., $id(t)=t$). Here the resolvent $u$ is related to the kernel $h$ in the same way as
in Example~3.3. Since $\textsl{K}_{P^*}[id](t)\neq 0$, there is no solution
to the Euler--Lagrange equation for functional $\mathcal{J}$.
The augmented Lagrangian $H$ of Theorem~3.7 is given by
$H(x_1,x_2,t) =(x_2+t)^2 -\lambda_0 tx_2$. Function
\begin{equation*}
y(t) = \left(\xi-1\right) \left( 1 +\int\limits_0^t u(t-\tau)d\tau \right)
\end{equation*}
is the solution to the Volterra integral equation $\textsl{K}_{P}[y](t)=(\xi-1)t$
of the first kind (see, e.g., Eq.~16, p.~114
of \cite{book:Polyanin}) and for $\lambda_0=2\xi$ satisfies
our optimality condition \eqref{eq:EL:ISO1}:
\begin{equation}
\label{eq:noc:ex2}
\textsl{K}_{P^*}\left[2\left(\textsl{K}_P[y](\tau)+\tau\right)
-2\xi \tau\right](t)=0.
\end{equation}
The solution of \eqref{eq:noc:ex2} subject to the given boundary conditions
depends on the particular choice for the kernel. For example, let
$h^{\alpha}(t-\tau)=e^{\alpha(t-\tau)}$. Then the solution of \eqref{eq:noc:ex2}
subject to the boundary conditions $y(0)=\xi-1$ and $y(1)=(\xi-1)(1-\alpha)$
is $y(t)=(\xi-1)(1-\alpha t)$ (cf. p.~15 of \cite{book:Polyanin}).
If $h^{\alpha}(t-\tau)=\cos\left(\alpha(t-\tau)\right)$, then the boundary
conditions are $y(0)=\xi-1$ and $y(1)=(\xi-1)\left(1+\alpha^2/2\right)$,
and the extremal is $y(t)=(\xi-1)\left(1+\alpha^2 t^2/2\right)$
(cf. p.~46 of \cite{book:Polyanin}).
\dl

Borrowing different kernels from book \cite{book:Polyanin},
many other examples of dynamic optimization problems can be
explicitly solved by application of the results of this section.

As particular cases of our generalized problem \eqref{eq:F:1}, \eqref{eq:F:3},
one obtains previously studied fractional isoperimetric problems
with Caputo derivatives.

\begin{cor}
{\rm (cf. \cite{MyID:207})}
Let $\y\in C^1([a,b];\R)$ be a minimizer to the functional
\begin{equation*}
\mathcal{I}(y)=\int\limits_a^b F\left(y(t),\dot{y}(t),\Dcl [y](t),t\right) dt
\end{equation*}
subject to an isoperimetric constraint of the form
\begin{equation*}
\mathcal{J}(y)=\int\limits_a^b
G\left(y(t),\dot{y}(t),\Dcl [y](t),t\right)dt=\xi
\end{equation*}
and boundary conditions
\begin{equation*}
y(a)=y_a,~~y(b)=y_b,
\end{equation*}
where $0<\alpha<\frac{1}{q}$ and functions $F$ and $G$ are such that
\begin{itemize}
\item $F,G\in C^1(\R^3\times [a,b];\R)$,
\item $t\mapsto\partial_2 F\left(y(t),\dot{y}(t),\Dcl [y](t),t\right)$,
$t\mapsto\partial_2 G\left(y(t),\dot{y}(t),\Dcl [y](t),t\right)$,
$$
\Irc\left[\partial_3 F\left(y(\tau),\dot{y}(\tau),\Dclt [y](\tau),\tau\right)\right],
$$
and
$$
\Irc\left[\partial_3 G\left(y(\tau),\dot{y}(\tau),\Dclt [y](\tau),\tau\right)\right]
$$
are continuously differentiable on $[a,b]$.
\end{itemize}
If $\y$ is  such that
\begin{multline*}
\partial_1 G\left(\y(t),\dot{\y}(t),\Dcl [\y](t),t\right)
-\frac{d}{dt}\left(\partial_2 G\left(\y(t),\dot{\y}(t),\Dcl [\y](t),t\right)\right)\\
+\Dr\left[
\partial_3 G\left(\y(\tau),\dot{\y}(\tau),\Dclt [\y](\tau),\tau\right)\right](t)\neq 0,
\end{multline*}
then there exists a constant $\lambda_0$ such that $\y$ satisfies
\begin{multline*}
\partial_1 H\left(\y(t),\dot{\y}(t),\Dcl [\y](t),t\right)-\frac{d}{dt}\left(
\partial_2 H\left(\y(t),\dot{\y}(t),\Dcl [\y](t),t\right)\right)\\
+\Dr\left[
\partial_3 H\left(\y(\tau),\dot{\y}(\tau),\Dclt [\y](\tau),
\tau\right)\right](t) = 0
\end{multline*}
with $H=F-\lambda_0 G$.
\end{cor}

Theorem~3.7 can be easily extended to $r$ subsidiary conditions of integral type.
Let $G_k$, $k=1,\dots, r$, be Lagrangians
\begin{equation*}
\fonction{G_k}{\R^4 \times [a,b]}{\R}{(x_1,x_2,x_3,x_4,t)}{G_k(x_1,x_2,x_3,x_4,t)}
\end{equation*}
of class $C^1$, and let
\begin{equation*}
\fonction{\mathcal{J}_k}{\mathcal{A}(y_a,y_b)}{\R}{y}{\displaystyle \int\limits_a^b
G_k(y(t),K_P[y](t),\dot{y}(t),B_P[y](t),t) \; dt,}
\end{equation*}
where $\dot{y}$ denotes the classical derivative of $y$, $K_P$ the generalized
fractional integral operator with a kernel belonging to $L^q(\Delta;\R)$,
and $B_P=K_P\circ\frac{d}{dt}$. Moreover, we assume that
\begin{itemize}
\item $K_{P^*}\left[\partial_2 G_k (y(\tau),K_P[y](\tau),
\dot{y}(\tau),B_P[y](\tau),\tau)\right]\in C([a,b];\R)$,
\item $t\mapsto\partial_3 G_k (y(t),K_P[y](t),
\dot{y}(t),B_P[y](t),t)\in C^1([a,b];\R)$,
\item $K_{P^*}\left[\partial_4 G_k (y(\tau),K_P[y](\tau),
\dot{y}(\tau),B_P[y](\tau),\tau)\right]\in C^1([a,b];\R)$.
\end{itemize}
Suppose that $\xi=(\xi_1,\dots,\xi_r)$ and define
$$
\mathcal{A}_{\xi}(y_a,y_b):=\left\{y\in\mathcal{A}(y_a,y_b):
\mathcal{J}_k[y]=\xi_k,\; k=1\dots,r\right\}.
$$
Next theorem gives necessary optimality condition for a minimizer
of functional $\mathcal{I}$ subject to $r$ isoperimetric constraints.

\begin{thm}
Let $\y$ be a minimizer of $\mathcal{I}$ in the class $\mathcal{A}_{\xi}(y_a,y_b)$.
If one can find functions $\eta_1,\dots,\eta_r\in\mathcal{A}(0,0)$
such that the matrix $A=\left(a_{kl}\right)$,
\begin{multline*}
a_{kl}:=\int\limits_a^b\; \left(\partial_1 G_k(\star_{\bar{y}})(t)
+K_{P^*}\left[\partial_2 G_k(\star_{\bar{y}})(\tau)\right](t)\right)\cdot\eta_l(\tau)\\
+\left(\partial_3 G_k(\star_{\bar{y}})(t)+K_{P^*}\left[
\partial_4 G_k(\star_{\bar{y}})(\tau)\right](t)\right)\cdot\dot{\eta_l}(t)\; dt,
\end{multline*}
has rank equal to $r$, then there exist $\lambda_1,\dots,\lambda_r\in\R$ such that $\y$ satisfies
\begin{multline}
\label{eq:EL:ISO2}
\frac{d}{dt}\left[\partial_3 H(\star_{\bar{y}})(t)\right]
+A_{P^*}\left[\partial_4 H(\star_{\bar{y}})(\tau)\right](t)\\
=\partial_1 H(\star_{\bar{y}})(t)+K_{P^*}\left[\partial_2
H(\star_{\bar{y}})(\tau)\right](t),
\quad t\in (a,b),
\end{multline}
where $(\star_{\bar{y}})(t)
=(\bar{y}(t),K_P[\bar{y}](t),\dot{\bar{y}}(t),B_P[\bar{y}](t),t)$ and
$H=F-\sum\limits_{k=1}^r\lambda_kG_k$.
\end{thm}

\begin{proof}
Let us define
\begin{equation*}
\fonction{\phi}{[-\varepsilon_0,\varepsilon_0]
\times\dots\times [-\varepsilon_r,\varepsilon_r]}{\R}{(h_0,h_1,
\dots,h_r)}{\mathcal{I}(\bar{y}+h_0\eta_0+h_1\eta_1+\dots+h_r\eta_r)}
\end{equation*}
and
\begin{equation*}
\fonction{\psi_k}{[-\varepsilon_0,\varepsilon_0]
\times\dots\times [-\varepsilon_r,\varepsilon_r]}{\R}{(h_0,h_1,\dots,h_r)}{\mathcal{J}_k(\bar{y}
+h_0\eta_0+h_1\eta_1+\dots+h_r\eta_r)-\xi_k.}
\end{equation*}
Observe that $\phi$ and $\psi_k$ are functions of class $C^1\left([-\varepsilon_0,\varepsilon_0]
\times\dots\times[-\varepsilon_r,\varepsilon_r];\R\right)$,
$A=\left(\left.\frac{\partial\psi_k}{\partial h_l}\right|_{0}\right)$ and that
$\psi_k (0,0,\dots,0)=0$. Moreover, because $\y$ is a minimizer of functional $\mathcal{I}$, we have
$$
\phi(0,\dots,0)\leq\phi(h_0,h_1,\dots,h_r).
$$
From the classical multiplier theorem,
there exist $\lambda_1,\dots,\lambda_r\in\R$ such that
\begin{equation}
\label{eq:ISOpf:1}
\nabla\phi_l(0,\dots,0)
+\sum\limits_{k=1}^r \lambda_k\nabla\psi_{k}(0,\dots,0)=0.
\end{equation}
From \eqref{eq:ISOpf:1} we can compute $\lambda_1,\dots,\lambda_r$,
independently of the choice of $\eta_0\in\mathcal{A}(0,0)$.
Finally, we arrive to \eqref{eq:EL:ISO2}
by the fundamental lemma of the calculus of variations.
\end{proof}


\subsection{Noether's theorem}
\label{sec:NTH:sing}

Emmy Noether's classical work \cite{Noether} from 1918 states that a conservation law
in variational mechanics follow whenever the Lagrangian function is invariant under
a one-parameter continuous group of transformations, that transform dependent and/or
independent variables. For instance, conservation of energy \eqref{eq:energ}
follows from invariance of the Lagrangian with respect to time-translations.
Noether's theorem not only unifies conservation laws but also suggests
a way to discover new ones. In this section we consider variational problems that depend
on generalized fractional integrals and derivatives. Following the methods used in
\cite{Cresson,MR2338631,book:Jost}, we apply Euler--Lagrange equations to formulate a generalized
fractional version of Noether's theorem without transformation of time. We start by introducing
the notions of generalized fractional extremal and one-parameter family of infinitesimal transformations.

\begin{definition}
A function $y\in C^1\left([a,b];\R\right)$ with $K_P[y]$ and $B_P[y]$ in $C\left([a,b];\R\right)$
that is a solution to equation \eqref{eq:EL:OCM} is said to be a generalized fractional extremal.
\end{definition}

We consider a one-parameter family of transformations of the form
$\hat{y}(t)=\phi(\theta,t,y(t))$, where $\phi$ is a map
\begin{equation*}
\fonction{\phi}{[-\varepsilon,\varepsilon]\times [a,b]
\times\R}{\R}{(\theta,t,x)}{\phi(\theta,t,x)}
\end{equation*}
of class $C^2$ such that $\phi(0,t,x)=x$.
Note that, using Taylor's expansion of $\phi(\theta,t,y(t))$ in a neighborhood of $0$, one has
\begin{equation*}
\hat{y}(t)=\phi(0,t,y(t))+\theta\frac{\partial}{\partial\theta}\phi(0,t,y(t))+o(\theta),
\end{equation*}
provided $\theta\in [-\varepsilon,\varepsilon]$. Moreover, having in mind that
$\phi(0,t,y(t))=y(t)$ and denoting
$\xi(t,y(t))=\frac{\partial}{\partial\theta}\phi(0,t,y(t))$, one has
\begin{equation}
\label{eq:Tr}
\hat{y}(t)=y(t)+\theta\xi(t,y(t))+o(\theta),
\end{equation}
so that the linear approximation to the transformation is
\begin{equation*}
\hat{y}(t)\approx y(t)+\theta\xi(t,y(t))
\end{equation*}
for $\theta\in [-\varepsilon,\varepsilon]$.
Now, let us introduce the notion of invariance.

\begin{definition}
We say that a Lagrangian $F$ is invariant under the one-parameter family
of infinitesimal transformations \eqref{eq:Tr}, where $\xi$ is such that
$t\mapsto\xi(t,y(t))\in C^1\left([a,b];\R\right)$ with
$K_P\left[\tau\mapsto\xi(\tau,y(\tau))\right],
B_P\left[\tau\mapsto\xi(\tau,y(\tau))\right]\in C\left([a,b];\R\right)$, if
\begin{equation}
\label{eq:CI:1}
F\left(y(t),K_P[y](t),\dot{y}(t),B_P[y](t),t\right)
= F\left(\hat{y}(t),K_P[\hat{y}](t),\dot{\hat{y}}(t),B_P[\hat{y}](t),t\right)
\end{equation}
for all $\theta\in [-\varepsilon,\varepsilon]$
and  all $y\in C^1\left([a,b];\R\right)$
with $K_P[y],B_P[y]\in C\left([a,b];\R\right)$.
\end{definition}

In order to state Noether's theorem in a compact form,
we introduce the following two bilinear operators:
\begin{equation}
\label{eq:BD:1}
\mathbf{D}[f,g]:=f\cdot A_{P^*}[g]+g\cdot B_P[f],
\end{equation}
\begin{equation}
\label{eq:BI:1}
\mathbf{I}[f,g]:=-f\cdot K_{P^*}[g]+g\cdot K_P[f].
\end{equation}

\begin{thm}
{\rm (Generalized Fractional Noether's Theorem)}
Let $F$ be invariant under the one-parameter family of infinitesimal
transformations \eqref{eq:Tr}. Then, for every generalized
fractional extremal $y$, the following equality holds:
\begin{multline}
\label{eq:NTH:1}
\frac{d}{dt}\left(\xi(t,y(t))\cdot\partial_3 F(\star_y)(t)\right)
+\mathbf{D}\left[\xi(t,y(t)),\partial_4 F(\star_y)(t)\right]\\
+\mathbf{I}\left[\xi(t,y(t)),\partial_2 F(\star_y)(t)\right]=0,
\end{multline}
$t\in (a,b)$, where $(\star_y)(t)=\left(y(t),K_P[y](t),\dot{y}(t),B_P[y](t),t\right)$.
\end{thm}

\begin{proof}
By equation \eqref{eq:CI:1} one has
\begin{equation*}
\left.\frac{d}{d\theta}\left[F\left(\hat{y}(t),K_P[\hat{y}](t),
\dot{\hat{y}}(t),B_P[\hat{y}](t),t\right)\right]\right|_{\theta=0}=0.
\end{equation*}
The usual chain rule implies that
\begin{multline*}
\Biggl.\partial_1 F(\star_{\hat{y}})(t)\cdot\frac{d}{d \theta}\hat{y}(t)
+\partial_2 F(\star_{\hat{y}})(t)\cdot\frac{d}{d \theta}K_P[\hat{y}](t)\\
+\partial_3 F(\star_{\hat{y}})(t)\cdot\frac{d}{d\theta}\dot{\hat{y}}(t)
+\partial_4 F(\star_{\hat{y}})(t)\cdot
\frac{d}{d\theta}B_P[\hat{y}](t)\Biggr|_{\theta=0}=0.
\end{multline*}
By linearity of $K_P$ and $B_P$, differentiating with respect
to $\theta$, and putting $\theta=0$,
\begin{multline*}
\partial_1 F(\star_y)(t)\cdot\xi(t,y(t))
+\partial_2 F(\star_y)(t)\cdot K_P[\tau\mapsto\xi(\tau,y(\tau))](t)\\
+\partial_3 F(\star_y)(t)\cdot\frac{d}{dt}\xi(t,y(t))
+\partial_4 F(\star_y)(t)\cdot B_P[\tau\mapsto\xi(\tau,y(\tau))](t)=0.
\end{multline*}
We obtain \eqref{eq:NTH:1} using the generalized Euler--Lagrange equation \eqref{eq:EL:OCM}
and \eqref{eq:BD:1}--\eqref{eq:BI:1}.
\end{proof}

\fudl

\example
Let $P=\langle a,t,b,\lambda,\mu\rangle$, $y\in C^1\left([a,b];\R\right)$,
$B_P[y]\in C\left([a,b];\R\right)$.
Consider Lagrangian $F\left(B_P[y](t),t\right)$
and transformations
\begin{equation}
\label{eq:Tr:2}
\hat{y}(t)=y(t)+\varepsilon c+o(\varepsilon),
\end{equation}
where $c$ is a constant. Then, we have
\begin{equation*}
F\left(B_P[y](t),t\right)=
F\left(B_P[\hat{y}](t),t\right).
\end{equation*}
Therefore, $F$ is invariant under
\eqref{eq:Tr:2} and the generalized fractional Noether's theorem asserts that
\begin{equation}
\label{eq:ex:NTH}
A_{P^*}[\partial_1 F\left(B_P[y](\tau),\tau\right)](t)=0,
\quad t\in (a,b),
\end{equation}
along any generalized fractional extremal $y$.
Notice that equation \eqref{eq:ex:NTH} can be written in the form
\begin{equation*}
\frac{d}{dt}\left(K_{P^*}[\partial_1
F\left(B_P[y](\tau),\tau\right)](t)\right)=0,
\end{equation*}
that is, quantity $K_{P^*}[\partial_1 F\left(B_P[y](\tau),\tau\right)]$
is conserved along all generalized fractional extremals and this quantity,
following the classical approach, can be called
a generalized fractional constant of motion.


\subsection{The multidimensional fractional calculus of variations}
\label{sec:SEV}

One can generalize results from Section~\ref{sec:fund} to the case of several variables.
In order to define the multidimensional generalized fractional variational problem,
we introduce the notion of generalized fractional gradient.

\begin{definition}
Let $n\in\N$ and $P=(P_1,\dots,P_n)$, $P_i=\langle a_i,t_,b_i,\lambda_i,\mu_i\rangle$.
The generalized fractional gradient of a function $f:\R^n\rightarrow\R$
with respect to the generalized fractional operator $T$ is defined by
\begin{equation*}
\nabla_{T_P}[f]:=\sum\limits_{i=1}^n e_i\cdot T_{P_i}[f],
\end{equation*}
where $\left\{e_i:i=1,\dots,n\right\}$
denotes the standard basis in $\R^n$.
\end{definition}

For $n\in\N$ let us assume that $P_i=\langle a_i,t_i,b_i,\lambda_i,\mu_i \rangle$
and $P=(P_1,\dots,P_n)$, $y:\R^n\rightarrow\R$, and $\zeta:\partial\Omega_n\rightarrow\R$
is a given function. Consider the following functional:
\begin{equation}
\label{eq:F:SEV}
\fonction{\mathcal{I}}{\mathcal{A}(\zeta)}{\R}{y}{\displaystyle \int\limits_{\Omega_n}
F(y(t),\nabla_{K_P}[y](t),\nabla[y](t),\nabla_{B_P}[y](t),t)\;dt,}
\end{equation}
where
$$
\mathcal{A}(\zeta):=\left\{y\in C^1(\bar{\Omega}_n;\R):\left.y\right|_{\partial\Omega_n}
=\zeta,~K_{P_i}[y],B_{P_i}[y]\in C(\bar{\Omega}_n;\R),i=1,\dots,n\right\},
$$
$\nabla$ denotes the classical gradient operator, and $\nabla_{K_P}$ and $\nabla_{B_P}$
are generalized fractional gradient operators such that $K_{P_i}$ is the generalized
partial fractional integral with kernel $k_i=k_i(t_i-\tau)$, $k_i\in L^1(0,b_i-a_i;\R)$,
and $B_{P_i}$ is the generalized partial fractional derivative of Caputo type satisfying
$B_{P_i}=K_{P_i}\circ\frac{\partial}{\partial t_i}$, $i=1,\dots,n$.
Moreover, we assume that $F$ is a Lagrangian
\begin{equation*}
\fonction{F}{\R\times\R^{3n}\times
\bar{\Omega}_n}{\R}{(x_1,x_2,x_3,x_4,t)}{F(x_1,x_2,x_3,x_4,t)}
\end{equation*}
of class $C^1$ and
\begin{itemize}
\item $K_{P_i^*}\left[\partial_{1+i}
F(y(\tau),\nabla_{K_P}[y](\tau),\nabla[y](\tau),\nabla_{B_P}[y](\tau),\tau)\right]
\in C(\bar{\Omega}_n;\R)$,

\item $t\mapsto\partial_{1+n+i} F(y(t),\nabla_{K_P}[y](t),
\nabla[y](t),\nabla_{B_P}[y](t),t)\in C^1(\bar{\Omega}_n;\R)$,

\item $K_{P_i^*}\left[\partial_{1+2n+i}
F(y(\tau),\nabla_{K_P}[y](\tau),\nabla[y](\tau),
\nabla_{B_P}[y](\tau),\tau)\right]\in C^1(\bar{\Omega}_n;\R)$,
\end{itemize}
$i=1,\ldots,n$.

The following theorem states that if a function minimizes functional \eqref{eq:F:SEV},
then it necessarily must satisfy \eqref{eq:EL:SEV}. This means that equation
\eqref{eq:EL:SEV} determines candidates to solve the problem
of minimizing functional \eqref{eq:F:SEV}.

\begin{thm}
{\rm (cf. \cite{FVC_Sev})}
Suppose that $\y\in\mathcal{A}(\zeta)$ is a minimizer of \eqref{eq:F:SEV}.
Then, $\y$ satisfies the following generalized Euler--Lagrange equation:
\begin{multline}
\label{eq:EL:SEV}
\partial_1 F(\star_{\y})(t)+\sum\limits_{i=1}^n
\Biggl(K_{P_i^*}[\partial_{1+i}F(\star_{\y})(\tau)](t)\\
-\frac{\partial}{\partial t_i}\left(\partial_{1+n+i}F(\star_{\y})(t)\right)
-A_{P_i^*}[\partial_{1+2n+i}F(\star_{\y})(\tau)](t)\Biggr)=0,
\end{multline}
$t\in\Omega_n$, where
$(\star_{\y})(t)=(\y(t),\nabla_{K_P}[\y](t),\nabla[\y](t),\nabla_{B_P}[\y](t),t)$.
\end{thm}

\fudl

\example
Consider a medium motion whose displacement may be described as a scalar function $y(t,x)$,
where $x=(x^1,x^2)$. For example, this function may represent the transverse displacement of a membrane.
Suppose that the kinetic energy $T$ and the potential energy $V$ of the medium are given by
\begin{equation*}
T\left(\frac{\partial y(t,x)}{\partial t}\right)
=\frac{1}{2}\int\rho(x)\left(\frac{\partial y(t,x)}{\partial t}\right)^2\;dx,
\end{equation*}
\begin{equation*}
V(y)=\frac{1}{2}\int k(x)\left|\nabla [y](t,x)\right|^2\;dx,
\end{equation*}
where $\rho(x)$ is the mass density and $k(x)$ is the stiffness,
both assumed positive. Then, the classical action functional is
\begin{equation*}
\mathcal{I}(y)=\frac{1}{2}\int\limits_{\Omega}\left(\rho(x)\left(
\frac{\partial y(t,x)}{\partial t}\right)^2
-k(x)\left|\nabla [y](t,x)\right|^2\right)\;dxdt.
\end{equation*}
We shall illustrate what are the Euler--Lagrange equations when
the Lagrangian density depends on generalized fractional derivatives.
When we have the Lagrangian with the kinetic term depending
on the operator $B_{P_1}$, with $P_1=\langle a_1,t,b_1,\lambda,\mu\rangle$,
then the fractional action functional has the form
\begin{equation}
\label{eq:F:SEVex}
\mathcal{I}(y)=\frac{1}{2}\int\limits_{\Omega_3}\left[\rho(x)\left(
B_{P_1}[y](t,x)\right)^2-k(x)\left|\nabla [y](t,x)\right|^2\right]\;dxdt.
\end{equation}
The fractional Euler--Lagrange equation
satisfied by an extremal of \eqref{eq:F:SEVex} is
\begin{equation*}
-\rho(x) A_{P_1^*}\left[B_{P_1}[y](\tau,s)\right](t,x)
-\nabla\left[k(s)\nabla[y](\tau,s)\right](t,x)=0.
\end{equation*}
If $\rho$ and $k$ are constants, then
$\rho A_{P_1^*}\left[B_{P_1}[y](\tau,s)\right](t,x)
+c^2\Delta[y](t,x)=0$,
where $c^2=k/\rho$, can be called \emph{the generalized time-fractional wave equation}.
Now, assume that the kinetic energy and the potential energy depend on $B_{P_1}$
and $\nabla_{B_P}$ operators, respectively, where $P=(P_{2},P_{3})$.
Then, the action functional for the system has the form
\begin{equation}
\label{eq:F:SEVex2}
\mathcal{I}(y)=\frac{1}{2}\int\limits_{\Omega_3}\left[
\rho\left(B_{P_1}[y](t,x)\right)^2
-k\left|\nabla_{B_P}[y](t,x)\right|^2\right]\;dxdt.
\end{equation}
The fractional Euler--Lagrange equation
satisfied by an extremal of \eqref{eq:F:SEVex2} is
\begin{equation*}
-\rho  A_{P_1^*}\left[B_{P_1}[y](\tau,s)\right](t,x)
+\sum\limits_{i=2}^3 A_{P_{i}^*}\left[B_{P_{i}}[y](\tau,s)\right](t,x)=0.
\end{equation*}
If $\rho$ and $k$ are constants, then
\begin{equation*}
A_{P_1^*}\left[B_{P_1}[y](\tau,s)\right](t,x)-c^2\left(
\sum\limits_{i=2}^3 A_{P_{i}^*}\left[kB_{P_{i}}[y](\tau,s)\right](t,x)\right)=0
\end{equation*}
can be called \emph{the generalized space- and time-fractional wave equation}.

\begin{cor}
Let $\alpha=(\alpha_1,\dots,\alpha_n)\in(0,1)^n$ and let
$\y\in C^1(\bar{\Omega}_n;\R)$ be a minimizer of the functional
\begin{equation*}
\mathcal{I}(y)=\int\limits_{\Omega_n} F(y(t),
\nabla_{I^{1-\alpha}}[y](t),\nabla[y](t),\nabla_{D^\alpha}[y](t),t)\;dt
\end{equation*}
satisfying $\left.y(t)\right|_{\partial \Omega_n}=\zeta(t)$,
where $\zeta:\partial\Omega_n\rightarrow\R$ is a given function,
\begin{equation*}
\nabla_{I^{1-\alpha}}=\sum\limits_{i=1}^n e_i\cdot\Ilcp,
~~\nabla_{D^\alpha}=\sum\limits_{i=1}^n e_i\cdot\Dclp,
\end{equation*}
$F$is of class $C^1$, and
\begin{itemize}
\item $\Ircp\left[\partial_{1+i} F(y(\tau),\nabla_{I^{1-\alpha}}[y](\tau),
\nabla[y](\tau),\nabla_{D^\alpha}[y](\tau),\tau)\right]$
is continuous on $\bar{\Omega}_n$,

\item $t\mapsto\partial_{1+n+i} F(y(t),\nabla_{I^{1-\alpha}}[y](t),
\nabla[y](t),\nabla_{D^\alpha}[y](t),t)$
is continuously differentiable on $\bar{\Omega}_n$,

\item $\Ircp\left[\partial_{1+2n+i} F(y(\tau),\nabla_{I^{1-\alpha}}[y](\tau),
\nabla[y](\tau),\nabla_{D^\alpha}[y](\tau),\tau)\right]$
is continuously differentiable on $\bar{\Omega}_n$.
\end{itemize}
Then, $\y$ satisfies the following fractional Euler--Lagrange equation:
\begin{multline*}
\partial_1 F(\y(t),\nabla_{I^{1-\alpha}}[\y](t),\nabla[\y](t),\nabla_{D^\alpha}[\y](t),t)\\
+\sum\limits_{i=1}^n\Biggl(\Ircp\left[\partial_{1+i}F(\y(\tau),
\nabla_{I^{1-\alpha}}[\y](\tau),\nabla[\y](\tau),\nabla_{D^\alpha}[\y](\tau),\tau)\right](t)\\
-\frac{\partial}{\partial t_i}\left(\partial_{1+n+i}F(\y(t),
\nabla_{I^{1-\alpha}}[\y](t),\nabla[\y](t),\nabla_{D^\alpha}[\y](t),t)\right)\\
+\Drp\left[\partial_{1+2n+i}F(\y(\tau),\nabla_{I^{1-\alpha}}[\y](\tau),
\nabla[\y](\tau),\nabla_{D^\alpha}[\y](\tau),\tau)\right](t)\Biggr)=0,
\end{multline*}
$t\in\Omega_n$.
\end{cor}

\fudl


\section{Conclusion}
\label{sec:conc}

During the last two decades, fractional differential
equations have increasingly attracted the attention of many
researchers: many mathematical problems in science
and engineering are represented by these kinds of equations.
Fractional calculus is a useful tool
for modeling complex behaviors of physical systems from diverse domains such as
mechanics, electricity, chemistry, biology, economics, and many others.
Science Watch of Thomson Reuters identified this area as an Emerging Research Front.
The Fractional Calculus of Variations (FCV) consider a new class of variational
functionals that depend on fractional derivatives and/or fractional integrals \cite{book:AD}.
Here we reviewed necessary conditions of optimality for problems of the FCV
with generalized operators \cite{Odz:PhD,MyID:226,FVC_Gen_Int,FVC_Sev,GreenThm,NoetherGen}.
The study of such variational problems is a subject
of strong current study because of its numerous applications.
The FCV is a fascinating and beautiful subject,
still in its childhood. It was created in 1996 in order to better
describe nonconservative systems in mechanics. The inclusion
of nonconservatism is extremely important from the point of
view of applications. Forces that do not store energy are always
present in real systems. They remove energy from the systems
and, as a consequence, Noether's conservation laws cease to
be valid. However, as we have proved here, it is still possible to obtain the validity of
Noether's principle using FCV. The new theory provides a more
realistic approach to physics, allowing us to consider nonconservative
systems in a natural way. To go further into the subject,
we refer the reader to \cite{book:AD,Odz:PhD} and references therein.


\section*{Acknowledgments}

This work was partially supported by Portuguese funds through the
\emph{Center for Research and Development in Mathematics and Applications} (CIDMA),
and \emph{The Portuguese Foundation for Science and Technology} (FCT),
within project PEst-OE/MAT/UI4106/2014.
The authors were also supported by EU funding under the
7th Framework Programme FP7-PEOPLE-2010-ITN,
grant agreement number 264735-SADCO, and project OCHERA,
PTDC/EEI-AUT/1450/2012, co-financed by FEDER under POFC-QREN
with COMPETE reference FCOMP-01-0124-FEDER-028894.




\begin{thebibliography}{99}
{\smaller

\bibitem{OmPrakashAgrawal3}
O. P. Agrawal,
Fractional variational calculus and the transversality conditions,
{\smallerit J. Phys. A} {\smallerbf 39} (33) (2006) 10375--10384.

\bibitem{OmPrakashAgrawal}
O. P. Agrawal,
Generalized variational problems and Euler-Lagrange equations,
{\smallerit Comput. Math. Appl.} {\smallerbf 59} (5) (2010) 1852--1864.

\bibitem{book:Baleanu}
D. Baleanu, K. Diethelm, E. Scalas\ and\ J. J. Trujillo,
{\smallerit  Fractional calculus},
Series on Complexity, Nonlinearity and Chaos, 3,
World Scientific Publishing Co. Pte. Ltd., Hackensack, NJ, 2012.

\bibitem{Cresson}
J. Cresson,
Fractional embedding of differential operators and Lagrangian systems,
{\smallerit J. Math. Phys.} {\smallerbf 48} (3) (2007) 033504, 34~pp.

\bibitem{MR2338631}
G. S. F. Frederico\ and\ D. F. M. Torres,
A formulation of Noether's theorem for fractional problems of the calculus of variations,
{\smallerit J. Math. Anal. Appl.} {\smallerbf 334} (2) (2007) 834--846.
{\tt arXiv:math/0701187}

\bibitem{hilfer}
R. Hilfer,
{\smallerit Applications of fractional calculus in physics},
World Sci. Publishing, River Edge, NJ, 2000.

\bibitem{book:Jost}
J. Jost\ and\ X. Li-Jost,
{\smallerit Calculus of variations},
Cambridge Studies in Advanced Mathematics, 64,
Cambridge Univ. Press, Cambridge, 1998.

\bibitem{Katugampola}
U. N. Katugampola,
New approach to a generalized fractional integral,
{\smallerit Appl. Math. Comput.} {\smallerbf 218} (3) (2011) 860--865.

\bibitem{Kilbas}
A. A. Kilbas, M. Saigo\ and\ R. K. Saxena,
Generalized Mittag-Leffler function and generalized fractional calculus operators,
{\smallerit Integral Transforms Spec. Funct.} {\smallerbf 15} (1) (2004) 31--49.

\bibitem{book:Kilbas}
A. A. Kilbas, H. M. Srivastava\ and\ J. J. Trujillo,
{\smallerit Theory and applications of fractional differential equations},
North-Holland Mathematics Studies, 204, Elsevier, Amsterdam, 2006.

\bibitem{book:Klimek}
M. Klimek,
{\smallerit On solutions of linear fractional differential equations of a variational type},
The Publishing Office of Czestochowa University of Technology, Czestochowa, 2009.

\bibitem{Lupa}
M. Klimek\ and\ M. Lupa,
Reflection symmetric formulation of generalized fractional variational calculus,
{\smallerit Fract. Calc. Appl. Anal.} {\smallerbf 16} (1) (2013) 243--261.

\bibitem{TM}
J. T. Machado, V. Kiryakova\ and\ F. Mainardi,
Recent history of fractional calculus,
{\smallerit Commun. Nonlinear Sci. Numer. Simul.} {\smallerbf 16} (3) (2011) 1140--1153.

\bibitem{book:AD}
A. B. Malinowska\ and\ D. F. M. Torres,
{\smallerit Introduction to the fractional calculus of variations},
Imp. Coll. Press, London, 2012.

\bibitem{Noether}
E. Noether,
Invariante variationsprobleme,
{\smallerit Nachr. v. d. Ges. d. Wiss. zu G\"{o}ttingen} (1918) 235--257.

\bibitem{Odz:PhD}
T. Odzijewicz,
{\smallerit Generalized fractional calculus of variations},
PhD Thesis, University of Aveiro, 2013.

\bibitem{Ja}
T. Odzijewicz,
{\smallerit Variable order fractional isoperimetric problem of several variables},
Advances in the Theory and Applications of Non-integer Order Systems,
Lecture Notes in Electrical Engineering 257, Springer, 2013, 133--139.

\bibitem{MyID:207}
T. Odzijewicz, A. B. Malinowska\ and\ D. F. M. Torres,
Fractional variational calculus with classical and combined Caputo derivatives,
{\smallerit Nonlinear Anal.} {\smallerbf 75} (3) (2012) 1507--1515.
{\tt arXiv:1101.2932}

\bibitem{MyID:226}
T. Odzijewicz, A. B. Malinowska\ and\ D. F. M. Torres,
Generalized fractional calculus with applications to the calculus of variations,
{\smallerit Comput. Math. Appl.} {\smallerbf 64} (10) (2012) 3351--3366.
{\tt arXiv:1201.5747}

\bibitem{FVC_Gen_Int}
T. Odzijewicz, A. B. Malinowska\ and\ D. F. M. Torres,
Fractional calculus of variations in terms of a generalized 
fractional integral with applications to physics,
{\smallerit Abstr. Appl. Anal.} {\smallerbf 2012} (2012) 871912, 24~pp.
{\tt arXiv:1203.1961}

\bibitem{FVC_Sev}
T. Odzijewicz, A. B. Malinowska\ and\ D. F. M. Torres,
Fractional calculus of variations of several independent variables,
{\smallerit European Phys. J.} {\smallerbf 222} (8) (2013) 1813--1826.
{\tt arXiv:1308.4585}

\bibitem{GreenThm}
T. Odzijewicz, A. B. Malinowska\ and\ D. F. M. Torres,
Green's theorem for generalized fractional derivatives,
{\smallerit Fract. Calc. Appl. Anal.} {\smallerbf 16} (1) (2013) 64--75.
{\tt arXiv:1205.4851}

\bibitem{NoetherGen}
T. Odzijewicz, A. B. Malinowska\ and\ D. F. M. Torres,
A generalized fractional calculus of variations,
{\smallerit Control Cybernet.} {\smallerbf 42} (2) (2013) 443--458.
{\tt arXiv:1304.5282}

\bibitem{livro:ortigueira}
M. D. Ortigueira,
{\smallerit Fractional calculus for scientists and engineers},
Lecture Notes in Electrical Engineering, 84, Springer, Dordrecht, 2011.

\bibitem{Podlubny}
I. Podlubny,
{\smallerit Fractional differential equations},
Mathematics in Science and Engineering, 198,
Academic Press, San Diego, CA, 1999.

\bibitem{book:Polyanin}
A. D. Polyanin\ and\ A. V. Manzhirov,
{\smallerit Handbook of integral equations},
CRC, Boca Raton, FL, 1998.

\bibitem{CD:Riewe:1996}
F. Riewe,
Nonconservative Lagrangian and Hamiltonian mechanics,
{\smallerit Phys. Rev. E (3)} {\smallerbf 53} (2) (1996) 1890--1899.

\bibitem{CD:Riewe:1997}
F. Riewe,
Mechanics with fractional derivatives,
{\smallerit Phys. Rev. E (3)} {\smallerbf 55} (3) (1997) 3581--3592.

\bibitem{book:Samko}
S. G. Samko, A. A. Kilbas\ and\ O. I. Marichev,
{\smallerit Fractional integrals and derivatives},
translated from the 1987 Russian original,
Gordon and Breach, Yverdon, 1993.

\label{pgCAbdio}
}

\end{thebibliography}
\end{document}